\documentclass[12pt,reqno]{amsart}
\usepackage{amsmath}
\usepackage{amssymb}
\usepackage{verbatim}
\usepackage{graphicx}
\usepackage[font=footnotesize]{caption}
\usepackage{subcaption}
\usepackage{epstopdf}
\usepackage{multirow}
\usepackage{url}
\usepackage{makecell}
\usepackage{tabularx}
\usepackage{longtable}
\usepackage{graphicx}
\usepackage{tikz}
\usepackage{pgfplots}
\usepgfplotslibrary{polar}
\usepackage{xcolor}
\usepackage{pgfplots}
\usepackage{pgfplotstable}
\allowdisplaybreaks

\usepackage{extarrows}
\usepackage{bm}
\usepackage[toc]{appendix}
\usepackage[capitalise]{cleveref}
\usepackage{color}
\allowdisplaybreaks

\usepackage[top=1.5in,bottom=1.5in,left=1in,right=1in]{geometry}

\newtheorem{lemma}{Lemma}

\newtheorem{theorem}[lemma]{Theorem}

\newtheorem{example}[lemma]{Example}

\numberwithin{equation}{section}
\numberwithin{lemma}{section}

\newcommand{\N}{\mathbb{N}}    
\newcommand{\NN}{\mathbb{N}_0} 
\newcommand{\R}{\mathbb{R}}    

\newcommand{\bo}{\mathcal{O}}

\newcommand{\xa}{x}
\newcommand{\ya}{y}

\newcommand{\p}{\partial}

\newcommand{\be}{ \begin{equation} }
	\newcommand{\ee}{ \end{equation} }

\newcommand{\ind}{\Lambda}

\newcommand{\bv}{{\bm v}}

\begin{document}
	
	\title[]{The convergence proof of the sixth-order compact 9-point FDM for the 2D transport problem}
	
\author{Qiwei Feng}
\address{Department of Mathematics, University of Pittsburgh, Pittsburgh, PA 15260 USA.}
\email{qif21@pitt.edu, qfeng@ualberta.ca}

\thanks{Research supported in part by  the Mathematics Research Center, Department of Mathematics, University of Pittsburgh, Pittsburgh, PA, USA (Qiwei Feng).
}

	\makeatletter \@addtoreset{equation}{section} \makeatother

	\begin{abstract}
It is widely acknowledged that the convergence proof of the error in the $l_{\infty}$ norm  of the high-order finite difference method (FDM) and finite element method (FEM) in 2D is challenging. In this paper, we derive the sixth-order compact 9-point FDM with the explicit stencil for the 2D transport problem  with the constant coefficient and the Dirichlet boundary condition in a unit square.  The proposed sixth-order FDM forms an M-matrix for the any mesh size $h$ employing the uniform Cartesian mesh.  The explicit formula of our FDM also enables us to construct the comparison function with the explicit expression to rigorously  prove the sixth-order convergence rate of  the maximum pointwise error  by the discrete maximum principle. Most importantly, we demonstrate that the sixth-order convergence proof is valid for any mesh size $h$. The numerical results are consistent with sixth-order accuracy in the $l_{\infty}$ norm. Our theoretical convergence proof is clear and the proposed sixth-order FDM is straightforward to be implemented, facilitating the reproduction of our numerical results.

\end{abstract}	


\maketitle

\pagenumbering{arabic}

		\section{Introduction}

The singularly-perturbed  problem arises in various applications, including the heat  transport process, the Navier-Stokes equation, the lid-driven cavity flow problem, and the  magnetohydrodynamic (MHD) duct flow problem \cite{Chiu2009,Hsieh2010,Morton1996,Roos2008,Tian2007}.

For  the  singularly-perturbed reaction-diffusion problem $ -\epsilon  \Delta u + au  = f$ with the variable function $a$ and the Dirichlet boundary condition in a unite square, \cite{Clavero2005,Gracia2006}  proved the almost second-order and third-order maximum pointwise accuracy of the central and  compact finite difference methods (FDMs) respectively using the piecewise uniform
Shishkin mesh,   the higher-order finite element method (FEM) was proposed in \cite{Li1999}  and also extended to solve the quasilinear problem. Furthermore, \cite{Cai2020,Lin2012} constructed a dual FEM and a balanced FEM respectively for  the  singularly-perturbed reaction-diffusion problem on a bounded polygonal or smooth
domain.
For the singularly-perturbed convection-diffusion/transport problem $ -\epsilon  \Delta u + au_x+bu_y=f $ with  the Dirichlet boundary condition in a rectangular domain, \cite{Chiu2009}  derived a dispersion-relation-preserving dual-compact upwind FDM for the constant coefficient $(a,b)$, \cite{Ge2011} built a fourth-order compact FDM on nonuniform grids for the variable coefficient $(a,b)$, \cite{Han2009}  proposed a tailored-finite-point method 
to achieve the high
accuracy in the coarse mesh for the variable coefficient $(a,b)$, 
\cite{Layton1990,Layton1993}  proved the second-order convergence rate in the maximum norm of the compact FDM for the positive constant coefficient $(a,b)$, 
\cite{Tian2007} established the higher-order compact exponential FEM with the uniform mesh for the variable coefficient.
 In particular,  a tailored finite point method for  $ -\epsilon  \Delta u + u_x=0 $ on a 2D unbounded domain was proposed in \cite{Han2008}.
For the general singularly-perturbed convection-diffusion problem $ -\epsilon_1  u_{xx}-\epsilon_2  u_{yy} + au_x+bu_y+cu=f $ with two parameters $\epsilon_1,\epsilon_2$, the variable coefficient $(a,b,c)$ and the Dirichlet boundary condition in a unite square, \cite{Nhan2021} proved the first-order convergence rate in the maximum norm of the upwind FDM using  the Bakhvalov mesh for $\epsilon_1 =\epsilon_2$, 
\cite{Tong2021} derived the fourth-order compact FDM for $\epsilon_1 =\epsilon_2$, 
\cite{Li2004} proposed the first-order  FEM with weighted basis functions discretized on the unstructured triangular mesh for $\epsilon_1 \ne \epsilon_2$.
For systems of singularly-perturbed equations with Dirichlet and/or Neumann boundary conditions in a square domain,
\cite{Hsieh2016} constructed the second-order FDM with the
 uniform mesh and piecewise-uniform Shishkin mesh,
\cite{Hsieh2010} provided a new upwind FDM  using
the uniform mesh.

In this paper,  we consider the 2D singularly-perturbed convection-diffusion/transport problem in $\Omega=(0,1)^2$ as follows
\be\label{Model:Original}
		\begin{cases}
 -\epsilon  \Delta u + au_x+bu_y  = f \text{ in } \Omega,\\	
 	u =g \text{ on }  \partial \Omega,\\
 0<a, \ 0<b,	\ 0<\epsilon \ll \bo(1), \\
 a=\text{const}, \quad b=\text{const}.
 \end{cases}
\ee
 To concisely present our sixth-order FDM, we define
\be\label{a:epsilon}
a_{\epsilon}:=a/\epsilon, \qquad b_{\epsilon}:=b/\epsilon, \qquad f_{\epsilon}:=f/\epsilon.
\ee
So \eqref{Model:Original} is equivalent to 
\be\label{Model:New}
\begin{cases}
	-  \Delta u + a_{\epsilon}u_x+b_{\epsilon}u_y  = f_{\epsilon} \text{ in } \Omega,\\	
	u =g \text{ on }  \partial \Omega.
\end{cases}
\ee
As we assume that $0<a$,  $0<b$, $0<\epsilon \ll \bo(1)$, we can obtain that
\be
a_{\epsilon}\ge 1 \quad \text{and} \quad b_{\epsilon}\ge 1.
\ee
To derive the sixth-order FDM, we also assume that
the solution $u\in C^8(\overline{\Omega})$ and the source term $f\in C^6(\overline{\Omega})$.

The organization of this paper is as follows:

In \cref{FDMs:all:points}, we propose the sixth-order FDM with the M-matrix property for any mesh size $h$ for the model problem \eqref{Model:Original} with $a\ge b$ in \cref{theorem:interior} in \cref{sec:case1}. Furthermore, we rigorously prove the  sixth-order convergence rate of the error in the $l_{\infty}$ norm  for any mesh size $h$ by the explicit formula of the comparison function and the discrete maximum principle. We extend the above result to \eqref{Model:Original} with $a\le b$ in \cref{theorem:2:interior} in \cref{sec:case2}. In \cref{numerical:test}, numerical experiments are provided to verify the sixth-order accuracy of the proposed FDM measured by the maximum pointwise norm.
In \cref{sec:contr}, the main contributions of this paper are summarized.
	\section{FDMs on uniform Cartesian grids for the 2D transport problem}\label{FDMs:all:points}

	Motivated by the method in \cite{Feng2022,Feng2024}, we propose the sixth-order FDM with the M-matrix property and construct the explicit expression of the comparison function to prove the convergence rate. Since the derivation of such a stencil is complicated and needs messy notations, we directly present the explicit expression of the sixth-order FDM to help readers understand the convergence proof.

We use use the uniform Cartesian grid with the mesh size $h$ for the 2D transport problem \eqref{Model:Original} as follows:
\be \label{xiyj}
\{ (\xa_i,\ya_j)  : 0 \le i,j \le N,\  \xa_i=i h,  \ \ya_j=j h,  \  h=1/N, \ N\in \N\},
\ee
where $\N$ denotes the set of all positive integers.

Recall that $u$ is the exact solution of \eqref{Model:Original}. We define that $u_{i,j}:=u(x_i,y_j)$, $u_h$  is the numerical solution  computed by our sixth-order FDM with the uniform mesh size $h$, and $(u_h)_{i,j}$ is the value of $u_h$ at the grid point $(x_i,y_j)$.
To present our sixth-order FDM clearly, we define that
\be\label{fmn}
f_\epsilon^{(m,n)}=\frac{\partial^{m+n} f_\epsilon(x,y)}{\partial x^{m} \partial y^{n} }\Big|_{(x,y)= (\xa_i,\ya_j)},
\ee
\be
\begin{split}\label{r1:r15}
	& 		r_{1}	:=a_\epsilon  + b_\epsilon, \qquad r_{2}:=a_\epsilon^2  + b_\epsilon^2,\qquad r_{3}:=a_\epsilon b_\epsilon,  \\
	&r_{4} := (14 a_\epsilon^4 + 33 a_\epsilon^3 b_\epsilon    + 42 a_\epsilon^2b_\epsilon^2 + 33 a_\epsilon b_\epsilon^3    + 10 b_\epsilon^4)/240,  \\
	&r_{5} := -(19 a_\epsilon^3 +  30a_\epsilon^2 b_\epsilon    + 21 a_\epsilon b_\epsilon^2)  /240, \qquad r_{6}=-(21 a_\epsilon^2 b_\epsilon +  30a_\epsilon b_\epsilon^2    + 19 b_\epsilon^3)/240,  \\	
	&r_{7} := (23 a_\epsilon^2 +  30a_\epsilon b_\epsilon   + 21 b_\epsilon^2) /240,\qquad 
	r_{8}:=(21 a_\epsilon^2  +  30a_\epsilon b_\epsilon  + 23 b_\epsilon^2) /240, \\
	 &r_{9}:=  r_1a_\epsilon(5a_\epsilon^3 + 9a_\epsilon^2b_\epsilon + 12a_\epsilon b_\epsilon^2 + 10b_\epsilon^3)/480, \quad
	r_{10} := -r_1a_\epsilon(9a_\epsilon^2 + 10a_\epsilon b_\epsilon + 11b_\epsilon^2)480,  \\
	 &r_{11}:= -r_1b_\epsilon(11a_\epsilon^2 + 10a_\epsilon b_\epsilon + 9b_\epsilon^2)/480, \quad r_{12} := (13a_\epsilon^3 + 23 a_\epsilon^2 b_\epsilon+ 21a_\epsilon b_\epsilon^2+ 11b_\epsilon^3)/480,  \\
	 & r_{13} := (11a_\epsilon^3 + 21 a_\epsilon^2 b_\epsilon+ 23a_\epsilon b_\epsilon^2+ 13b_\epsilon^3)/480,  \\
	 & r_{14} := (8 a_\epsilon ^4 + 33 a_\epsilon ^3 b_\epsilon   + 42 a_\epsilon ^2 b_\epsilon ^2  + 33 a_\epsilon b_\epsilon ^3  + 16 b_\epsilon ^4)/240, \\
	&r_{15} :=(5 a_\epsilon ^5 - 22 a_\epsilon ^4 b_\epsilon  + 27 a_\epsilon ^3 b_\epsilon ^2 + 16 a_\epsilon ^2 b_\epsilon ^3 + 4 a_\epsilon  b_\epsilon ^4 + 42 b_\epsilon ^5)/480, 
\end{split}	
\ee
where $a_\epsilon, b_\epsilon, f_\epsilon$ are defined in \eqref{a:epsilon}.

\subsection{The sixth-order FDM for \bm{$a\ge b $}}\label{sec:case1}
In this section, we derive the sixth-order 9-point compact FDM for the model problem \eqref{Model:Original} with $a\ge b $. We construct the explicit expression of the sixth-order FDM with the M-matrix property. We also construct the explicit expression of the comparison function to theoretically prove the sixth-order convergence rate in the $l_{\infty}$ norm by the discrete maximum principle.
Recall that  $ r_1,\dots r_{15}$ are defined in \eqref{r1:r15}.
To help readers to implement our sixth-order FDM easily, we define that (LHS and RHS represent the left-hand side and right-hand side, respectively):\\
\textbf{coefficients of $h^0$ of the LHS of the FDM:}
\begin{align}\label{c:h:0}
	& c_{-1,\pm 1,0}:=c_{1,\pm 1,0}:=-1, \qquad c_{\pm 1,0,0}:=c_{0,\pm 1,0}:=-4, \qquad c_{0,0,0}:=20,
\end{align}
\textbf{coefficients of $h$ of the  LHS of the FDM:}
\be
\begin{split}\label{c:h:1}
	& c_{-1, -1, 1} := -r_1, \quad c_{-1, 0, 1} := -2(2a_{\epsilon} + b_{\epsilon}), \quad c_{-1, 1, 1} := -a_{\epsilon},\quad c_{0, -1, 1} := -2(a_{\epsilon} +2b_{\epsilon}),  \\
	& c_{0, 0, 1} := 10r_1, \quad c_{0, 1, 1} := -2a_{\epsilon}, \quad c_{1, -1, 1} := -b_{\epsilon}, \quad c_{1, 0, 1} := -2b_{\epsilon}, \quad c_{1, 1, 1} := 0,
\end{split}
\ee
\textbf{coefficients of $h^2$ of the  LHS of the FDM:}
\begin{align}\label{c:h:2}
	& c_{-1, -1, 2} := -r_1^2/2, \quad c_{-1, 0, 2} :=- ( 41a_{\epsilon}^2 + 40a_{\epsilon}b_{\epsilon} +11b_{\epsilon}^2)/20, \quad c_{-1, 1, 2} := -a_{\epsilon}^2/2, \notag  \\
	&  c_{0, -1, 2} := -(  11a_{\epsilon}^2+ 40a_{\epsilon}b_{\epsilon} +41b_{\epsilon}^2)/20, \quad c_{0, 0, 2} := (21r_2 + 25a_{\epsilon}b_{\epsilon})/5,  \\
	& c_{0, 1, 2} := -(11a_{\epsilon}^2 + b_{\epsilon}^2)/20, \quad c_{1, -1, 2} := -b_{\epsilon}^2/2, \quad c_{1, 0, 2} := -(a_{\epsilon}^2+ 11b_{\epsilon}^2)/20, \quad c_{1, 1, 2} := 0, \notag
\end{align}
\textbf{coefficients of $h^3$ of the  LHS of the FDM:}
\begin{align}\label{c:h:3}
	& c_{-1, -1, 3} := -r_1^3/6, \quad c_{-1, 0, 3} :=- (86a_\epsilon^3 + 123a_\epsilon^2b_\epsilon + 66a_\epsilon b_\epsilon^2 + 13b_\epsilon^3)/120, \notag\\
	& c_{-1, 1, 3} := -a_\epsilon^3/6, \quad c_{0, -1, 3} := -(13a_\epsilon^3 + 66a_\epsilon^2b_\epsilon  + 123 a_\epsilon b_\epsilon^2 + 86b_\epsilon^3)/120,  \\
	& c_{0, 0, 3} := 19 (a_\epsilon^3+b_\epsilon^3)/15 + 21 r_1r_3/10, \quad c_{0, 1, 3} := -(13a_\epsilon^3 + 3a_\epsilon b_\epsilon^2)/120, \notag \\
	& c_{1, -1, 3} := -b_\epsilon^3/6, \quad c_{1, 0, 3} := -(3a_\epsilon^2b_\epsilon  + 13b_\epsilon^3)/120, \quad c_{1, 1, 3} := 0, \notag
\end{align}
\textbf{coefficients of $h^4$ of the  LHS of the FDM:}
\be
\begin{split}\label{c:h:4}
	& c_{-1, -1, 4} := -r_1^4/24, \quad c_{-1, 0, 4} := -(23a_\epsilon^4 + 43a_\epsilon^3 b_\epsilon + 33a_\epsilon^2b_\epsilon^2 + 13a_\epsilon b_\epsilon^3)/120, \\
	& c_{-1, 1, 4} := -a_\epsilon^4/24, \quad c_{0, -1, 4} := -(2a_\epsilon^4 + 13a_\epsilon^3 b_\epsilon + 33a_\epsilon^2b_\epsilon^2 + 43a_\epsilon b_\epsilon^3 + 21b_\epsilon^4)/120, \\
	& c_{0, 0, 4} := (37a_\epsilon^4+29b_\epsilon^4)/120 + 19r_2r_3/30 + 4r_3^2/5, \quad c_{0, 1, 4} := (b_\epsilon^4- a_\epsilon^4)/60,\\
	&  c_{1, -1, 4} := -b_\epsilon^4/24, \quad c_{1, 0, 4} := 0, \quad c_{1, 1, 4} := 0,
\end{split}
\ee
\textbf{coefficients of $h^5$ of the  LHS of the FDM:}
\begin{align}\label{c:h:5}
	& c _{-1, -1, 5} := (b_\epsilon^5 -5 a_\epsilon^5 - 20 a_\epsilon^4 b_\epsilon - 39 a_\epsilon^3 b_\epsilon^2 - 41 a_\epsilon^2 b_\epsilon^3 - 16 a_\epsilon b_\epsilon^4 )/480, \notag  \\
	& c _{-1, 0, 5} := -(9 a_\epsilon^5 + 23 a_\epsilon^4 b_\epsilon + 23 a_\epsilon^3 b_\epsilon^2 + 12 a_\epsilon^2 b_\epsilon^3 + 4 a_\epsilon b_\epsilon^4 + b_\epsilon^5)/240,  \notag \\
	& c _{-1, 1, 5} := (b_\epsilon^5 -5 a_\epsilon^5+ a_\epsilon^3 b_\epsilon^2 - a_\epsilon^2 b_\epsilon^3 + 4 a_\epsilon b_\epsilon^4 )/480, \\
	& c _{0, -1, 5} := - r_1b_\epsilon (a_\epsilon^2 + a_\epsilon b_\epsilon + 2 b_\epsilon^2)  (4 a_\epsilon + 5 b_\epsilon)/240,  \notag \\
	& c _{0, 0, 5} := (14 a_\epsilon^5 + 37 a_\epsilon^4 b_\epsilon + 55 a_\epsilon^3 b_\epsilon^2 + 55 a_\epsilon^2 b_\epsilon^3 + 33 a_\epsilon b_\epsilon^4 + 10 b_\epsilon^5)/240,  \notag \\
	& c _{0, 1, 5} := 0, \quad c _{1, -1, 5} := 0, \quad c _{1, 0, 5} := 0, \quad c _{1, 1, 5} := 0, \notag 
\end{align}
\textbf{coefficients of $h^6$ of the  LHS of the FDM:}
\begin{align}\label{c:h:6}
	& c_{-1, -1, 6} := -(5a_\epsilon + 4b_\epsilon)\alpha/a_\epsilon, \quad c_{-1, 0, 6} := 4(b_\epsilon - a_\epsilon)\alpha/a_\epsilon, \quad c_{-1, 1, 6} := -\alpha,   \notag \\
	& c_{0, -1, 6} := 0, \quad c_{0, 0, 6} := 20\alpha, \quad c_{0, 1, 6} := -4\alpha, \quad c_{1, -1, 6} := -\alpha, \quad c_{1, 0, 6} := -4\alpha, \\
	& c_{1, 1, 6} := -\alpha, \quad \alpha:=a_\epsilon^2r_1(5a_\epsilon^3 + 9a_\epsilon^2b_\epsilon + 12a_\epsilon b_\epsilon^2 + 10b_\epsilon^3)/1920,		\notag
\end{align}
\textbf{the  RHS of the FDM:}
\be
	\begin{split}\label{F:ij}
	F_{i,j}	&:=6f_{_\epsilon}+3h	r_{1} f_{_\epsilon}+ h^2 ((21r_2 + 30r_3)f_{_\epsilon}/20 - 	\bv \cdot \nabla f_{_\epsilon}/2 + \Delta f_{_\epsilon}/2) \\
	& \quad +{h^3r_{1}}	((11 r_2/40 + r_3/4 ) f_{_\epsilon} -  (	\bv \cdot \nabla f_{_\epsilon} - \Delta f_{_\epsilon})/4)\\
	&\quad + h^4 ( r_{4} f_{_\epsilon} +r_{5} f_{_\epsilon}^{(1, 0)}  +r_{6} f_{_\epsilon}^{(0, 1)}+r_{7} f_{_\epsilon}^{(2, 0)}+ r_{3}  f_{_\epsilon}^{(1, 1)}/15+ r_{8} f_{_\epsilon}^{(0, 2)}-  2\zeta )\\
	&\quad + h^5 (  r_{9} f_{_\epsilon} +r_{10} f_{_\epsilon}^{(1, 0)}  +r_{11} f_{_\epsilon}^{(0, 1)}+r_{12} f_{_\epsilon}^{(2, 0)}+ r_{1}r_{3}  f_{_\epsilon}^{(1, 1)}/30+ r_{13} f_{_\epsilon}^{(0, 2)} \\
	&\quad \qquad -  r_{1}\zeta+ r_{1}( \Delta^2  f_{_\epsilon} +2f_{_\epsilon}^{(2, 2)} )/120),
\end{split}
\ee
where
\be\label{zeta}
\bv:= (a_\epsilon,  b_\epsilon),\qquad \zeta:=(a_\epsilon f_\epsilon^{(3, 0)}  +  2b_\epsilon f_\epsilon^{(2, 1)} +  2a_\epsilon f_\epsilon^{(1, 2)} + b_\epsilon f_\epsilon^{(0, 3)})/60,
\ee
 $r_1,\dots r_{15}$ are defined in \eqref{r1:r15} and  each $f_\epsilon^{(m,n)}$ is defined in \eqref{fmn}.

Using the above \eqref{c:h:0}--\eqref{zeta}, we prove the sixth-order convergence rate of our FDM in the following \cref{theorem:interior}.
\begin{theorem}\label{theorem:interior} Assume that the solution $u$, the source term $f$, and the constant coefficient $(a,b)$ of \eqref{Model:Original} satisfy that $u\in C^8(\overline{\Omega})$, $f\in C^6(\overline{\Omega})$, and $a\ge b$. Let the grid point $(x_i,y_j)\in \Omega$ and define the following 9-point compact FDM:
	\be\label{FDMs:u:interior}
		\mathcal{L}_h (u_h)_{i,j} :=\frac{1}{h^2}\sum_{k,\ell=-1}^1 C_{k,\ell} (u_h)_{i+k,j+\ell}=F_{i,j}\Big|_{(x,y)=(x_i,y_j)} \quad \text{with} \quad C_{k,\ell}:=\sum_{p=0}^6 c_{k,\ell ,p}h^p,
	\ee
where $c_{k,\ell,p}$ is defined in \eqref{c:h:0}--\eqref{c:h:6} and $F_{i,j}$ is defined in \eqref{F:ij}. Then
\be\label{sign:cd}
C_{k,\ell}\le  0 \quad \text{if} \quad (k,\ell)\ne (0,0), \quad   C_{0,0}>  0, \quad  \text{for any mesh size } h>0, \quad \text{and} \quad \sum_{k,\ell=-1}^1 C_{k,\ell}=0.
\ee
Furthermore,   $	\mathcal{L}_h (u_h)_{i,j}$ in \eqref{FDMs:u:interior}   with $u=g$ on $\partial \Omega$  forms an M-matrix for any mesh size $h>0$,  and
	\be\label{convergence:order:6}
\|u-u_h\|_{\infty}\le Ch^6, \quad  \text{for any mesh size } h>0,
	\ee
	where  $C$ is independent of  $h$.
\end{theorem}

\begin{proof}
\textbf{Step 1 (the sixth	order of consistency):}		We define 
\be \label{ind}
\ind_{M+1}:=\{ (m,n)\in \NN^2 \; : \;  m+n\le M+1 \}, \qquad M+1\in \NN,\qquad \NN:=\N\cup\{0\},
\ee
and $\N$ denotes the set of all positive integers.
Then the Taylor expansion at the grid point $(\xa_i,\ya_j)\in \Omega$ with the Lagrange remainder can be written as 
\[
u(x_i+x,y_j+y)
=
\sum_{(m,n)\in \ind_{M+1}}
\frac{\partial^{m+n}u(\xa_i,\ya_j) }{\partial x^m \partial y^n}\frac{x^my^n}{m!n!}+\sum_{\substack{0\le m,n \le M+2 \\ m+n=M+2}} \frac{\partial^{M+2}u(\xi,\eta) }{\partial x^m \partial y^{n}}\frac{x^my^{n}}{m!n!},
\]
where $\xi \in (x_i-|x|,x_i+|x|)$ and $\eta \in (y_j-|y|,y_j+|y|)$. Choose $M=6$, and $x,y=0,\pm h,$
\be \label{taylor:u}
u(x_i+x,y_j+y)
=
\sum_{(m,n)\in \ind_{7}}
\frac{\partial^{m+n}u(\xa_i,\ya_j) }{\partial x^m \partial y^n}\frac{x^my^n}{m!n!}+c_1h^8, 
\ee
where $c_1$ only depends on $\frac{\partial^{8}u(\xi,\eta) }{\partial x^m \partial y^{n}}$ with $0\le m,n \le 8$,  $m+n=8$, $\xi \in (x_i-|x|,x_i+|x|)$ and $\eta \in (y_j-|y|,y_j+|y|)$. As $u\in C^8(\overline{\Omega})$, $|c_1|<\infty$. 
By the definition of $	\mathcal{L}_h (u_h)_{i,j} $ in \eqref{FDMs:u:interior}, we also define that
\be\label{L:h:u:1}
	\mathcal{L}_h (u)_{i,j} :=\frac{1}{h^2}\sum_{k,\ell=-1}^1 C_{k,\ell} u_{i+k,j+\ell}=\frac{1}{h^2}\sum_{k,\ell=-1}^1 C_{k,\ell} u(\xa_i+kh,\ya_j+\ell),
\ee
for any $(x_i,y_j)\in \Omega$.
	Let 
	\be\label{L:u}
	\mathcal{L} u:=-  \Delta u + a_{\epsilon}u_x+b_{\epsilon}u_y.
	\ee
Plug the defined $C_{k,\ell}$ in \eqref{FDMs:u:interior} with \eqref{c:h:0}--\eqref{c:h:6} and \eqref{taylor:u} into \eqref{L:h:u:1}, after the simplification, we obtain
\be\label{L:h:u:2}
\mathcal{L}_h (u)_{i,j} :=\sum_{p=0}^5C_p h^p\Big|_{(x,y)=(x_i,y_j)}+c_2h^6, \quad \text{for any }  h>0,
\ee
where $0<c_2$ is independent of $h$,
	\begin{align}\label{C0:C5}
	&C_0=6\mathcal{L} u,\qquad C_1=3r_{1} \mathcal{L} u, \notag\\
	&C_2= (21r_2 + 30r_3)\mathcal{L} u/20 - 	(\bv \cdot \nabla \mathcal{L} u - \Delta \mathcal{L} u) /2, \notag\\
	&C_3=\tfrac{r_{1}}{40}	(11 r_2 + 10 r_3 ) \mathcal{L} u - \tfrac{r_{1}}{4} (	\bv \cdot \nabla \mathcal{L} u - \Delta \mathcal{L} u),\notag\\
	&C_4=  r_{4} \mathcal{L} u +r_{5} \mathcal{L} u_x  +r_{6} \mathcal{L} u_y+r_{7} \mathcal{L} u_{xx}+ r_{3} \mathcal{L} u_{xy}/15+ r_{8} \mathcal{L} u_{yy} \\
	&\qquad -   (a_\epsilon \mathcal{L} u_{xxx}  + 2 b_\epsilon \mathcal{L} u_{xxy} +  2a_\epsilon \mathcal{L} u_{xyy} + b_\epsilon \mathcal{L} u_{yyy})/30,  \notag \\
	&C_5=   r_{9} \mathcal{L} u +r_{10} \mathcal{L} u_{x}  +r_{11} \mathcal{L} u_{y}+r_{12} \mathcal{L} u_{xx}+ r_{1}r_{3}  \mathcal{L} u_{xy}/30+ r_{13} \mathcal{L} u_{yy} \notag\\
	&\qquad -  r_{1}(a_\epsilon \mathcal{L} u_{xxx}  +  2b_\epsilon \mathcal{L} u_{xxy} +  2a_\epsilon \mathcal{L} u_{xyy} + b_\epsilon \mathcal{L} u_{yyy})/60 \notag\\
	&\qquad + r_{1}( \Delta^2  \mathcal{L} u +2 \mathcal{L} u_{xxyy}  )/120.\notag
\end{align}	
From \eqref{L:u}, we observe that $	\mathcal{L} u=f_{\epsilon}$ and  $	\mathcal{L} \tfrac{\p ^{m+n}u}{\p x^m \p y^n}=\tfrac{\p ^{m+n} f_{\epsilon}}{\p x^m \p y^n} $. Now \eqref{F:ij}, $\mathcal{L}_h(u_h)$ in \eqref{FDMs:u:interior}, \eqref{L:h:u:2}, and \eqref{C0:C5} imply that
\be\label{truncation:error:h:6}
|\mathcal{L}_h (u_h)_{i,j}-\mathcal{L}_h (u)_{i,j}|	\le c_3 h^6,  \quad \text{for any }  h>0 \text{ and } (x_i,y_j)\in \Omega,
\ee	
where $c_3=\max_{(x_i,y_i)\in \Omega} |c_2|$, and $c_3$ is independent of $h$.
	
\textbf{Step 2 (the M-matrix property):}
By \eqref{c:h:0}--\eqref{c:h:6}, $a_\epsilon \ge 1$, $b_\epsilon \ge 1$, and $a_\epsilon \ge b_\epsilon$, it is straightforward to verify that
\be	\label{ckl:p:0:5}
\begin{split}
& c_{k,\ell,0}<  0 \quad \text{if} \quad (k,\ell)\ne (0,0), \quad \text{and} \quad  c_{0,0,0}>  0,\\	
& c_{k,\ell,p}\le  0 \quad \text{if} \quad (k,\ell)\ne (0,0), \quad \text{and} \quad  c_{0,0,p}\ge  0, \quad \text{for} \quad p=1,2,3,4,6,\\
& c_{k,\ell,5}\le  0 \quad \text{if} \quad (k,\ell)\notin\{(-1,-1), (-1,1), (0,0)\}, \quad \text{and} \quad  c_{0,0,5}\ge  0.
\end{split}
\ee
We set that
\[
a_\epsilon =b_\epsilon+\omega, \quad \text{with} \quad \omega\ge 0, \quad a_\epsilon \ge 0, \quad b_\epsilon \ge 0.
\]
Then $ c _{-1, -1, 5}$ and $ c _{-1, 1, 5}$ in \eqref{c:h:5} lead to
\be\label{ckl:p5}
\begin{split}
	& c _{-1, -1, 5} = -(120 b_\epsilon ^5 + 320 \omega b_\epsilon ^4 + 328 \omega^2 b_\epsilon ^3 + 169 \omega^3 b_\epsilon ^2 + 45  \omega^4 b_\epsilon + 5\omega^5)/480
	\le 0,   \\
	& c _{-1, 1, 5} = -(20 \omega b_\epsilon ^4 + 48\omega^2 b_\epsilon ^3 + 49\omega^3 b_\epsilon ^2 + 25\omega^4 b_\epsilon  + 5\omega^5)/480
	\le 0.
\end{split}	
\ee
So, $C_{k,\ell}$ in \eqref{FDMs:u:interior} satisfies that
\be\label{M:Matrix}
C_{k,\ell}\le  0 \quad \text{if} \quad (k,\ell)\ne (0,0), \quad \text{and} \quad  C_{0,0}>  0, \quad \text{for any } h>0.
\ee
On the other hand, by the direct calculation, $C_{k,\ell}$ in \eqref{FDMs:u:interior} also satisfies that
\be \label{Ckl:sum}
\sum_{k,\ell=-1}^1 C_{k,\ell}=0, \quad \text{for any } h>0.
\ee

According to \cite{LiZhang2020}, an M-matrix is a square matrix with positive diagonal entries, non-positive off-diagonal entries, and non-negative row sums for all rows, with at least one row sum strictly positive. So
the sign condition \eqref{M:Matrix}  and summation condition \eqref{Ckl:sum} with the Dirichlet boundary condition $u=g$ on $\partial \Omega$ indicate that \eqref{FDMs:u:interior} yields an M-matrix for any mesh size $h$.

\textbf{Step 3 (construct the comparison function):}
	We define that
\be\label{L:tilde}
	\mathcal{\tilde{L}}_h:=	-\mathcal{L}_h, \qquad  \theta(x,y):=(x-a_\epsilon)^2+(y-b_\epsilon)^2,
\ee
where $\theta(x,y)$ is the comparison function. 
Then
\[
	\mathcal{\tilde{L}}_h(\theta)_{i,j} =-\frac{1}{h^2}\sum_{k,\ell=-1}^1 C_{k,\ell} \theta_{i+k,j+\ell}=-\frac{1}{h^2}\sum_{k,\ell=-1}^1 C_{k,\ell} \theta(x_i+kh,y_j+\ell h), \quad (x_i,y_j)\in \Omega.
\]
By the direct calculation, 
\be\label{L:tilde:2}
\mathcal{\tilde{L}}_h(\theta)_{i,j}=\sum_{p=0}^6 \phi_ph^p, \quad \text{for any }  h>0 \text{ and } (x_i,y_j)\in \Omega,
\ee
where 
\begin{align}
	&\phi_0=12(a_\epsilon^2 - a_\epsilon x_i + b_\epsilon^2 - b_\epsilon y_j + 2), \notag\\
	&\phi_1=6(a_\epsilon + b_\epsilon)(a_\epsilon^2 - a_\epsilon x_i + b_\epsilon^2 - b_\epsilon y_j + 2), \notag\\
	&\phi_2=(21a_\epsilon^4 + (30b_\epsilon - 21 x_i)a_\epsilon^3 + (42b_\epsilon^2 - (30 x_i + 21 y_j)b_\epsilon + 52)a_\epsilon^2 \notag \\
	&\qquad + (30b_\epsilon^3 - (21 x_i + 30 y_j)b_\epsilon^2 + 60b_\epsilon)a_\epsilon + 21b_\epsilon^4 - 21 y_jb_\epsilon^3 + 52b_\epsilon^2)/10,	\notag\\
	&\phi_3=(11a_\epsilon^5 + (21b_\epsilon - 11 x_i)a_\epsilon^4 + (32b_\epsilon^2 - 21b_\epsilon x_i - 11b_\epsilon y_j + 32)a_\epsilon^3\notag \\
	& \qquad  + (32b_\epsilon^3 - 21b_\epsilon^2 x_i - 21b_\epsilon^2 y_j + 52b_\epsilon)a_\epsilon^2 + (21b_\epsilon^4 - 11b_\epsilon^3 x_i - 21b_\epsilon^3 y_j + 52b_\epsilon^2)a_\epsilon \notag \\
	& \qquad+ 11b_\epsilon^5 - 11 y_jb_\epsilon^4 + 32b_\epsilon^3)/20, \notag \\
	&\phi_4=(14a_\epsilon^6 + (33b_\epsilon - 14 x_i)a_\epsilon^5 + (56b_\epsilon^2 - 33b_\epsilon x_i - 14b_\epsilon y_j + 47)a_\epsilon^4 + (66b_\epsilon^3 - 42b_\epsilon^2 x_i \notag \\
&\qquad - 33b_\epsilon^2 y_j + 96b_\epsilon)a_\epsilon^3 + (52b_\epsilon^4 - 33b_\epsilon^3 x_i - 42b_\epsilon^3 y_j + 126b_\epsilon^2)a_\epsilon^2 + (33b_\epsilon^5 - 10b_\epsilon^4 x_i \notag \\
&\qquad - 33b_\epsilon^4 y_j + 96b_\epsilon^3)a_\epsilon + 10b_\epsilon^6 - 10 y_jb_\epsilon^5 + 39b_\epsilon^4)/120, \notag	\\
&	\phi_5=(5a_\epsilon^7 + (14b_\epsilon - 5 x_i)a_\epsilon^6 + (26b_\epsilon^2 - 14b_\epsilon x_i - 5b_\epsilon y_j + 19)a_\epsilon^5 + (36b_\epsilon^3 - 21b_\epsilon^2 x_i  \notag \\
&\qquad - 14b_\epsilon^2 y_j + 47b_\epsilon)a_\epsilon^4 + (31b_\epsilon^4 - 22b_\epsilon^3 x_i - 21b_\epsilon^3 y_j + 74b_\epsilon^2)a_\epsilon^3 + (22b_\epsilon^5 - 10b_\epsilon^4 x_i\notag \\
&\qquad  - 22b_\epsilon^4 y_j + 76b_\epsilon^3)a_\epsilon^2 + (10b_\epsilon^6 - 10b_\epsilon^5 y_j + 39b_\epsilon^4)a_\epsilon + 9b_\epsilon^5)/240, \notag\\
	&\phi_6=a_\epsilon(7a_\epsilon + b_\epsilon)(5a_\epsilon^3 + 9a_\epsilon^2b_\epsilon + 12a_\epsilon b_\epsilon^2 + 10b_\epsilon^3)(a_\epsilon + b_\epsilon)/480.\notag
\end{align}
As $0<x_i<1$, $0<y_j<1$, $1\le a_\epsilon$, and $1\le b_\epsilon$, we infer that
\be\label{psik}
\phi_k\ge \psi_k,\qquad k=0,\dots, 5, \quad \text{and} \quad \phi_6\ge 0,
\ee
where 
\begin{align}
&\psi_0=12(a_\epsilon^2 - a_\epsilon  + b_\epsilon^2 - b_\epsilon + 2), \notag\\
&\psi_1=6(a_\epsilon + b_\epsilon)(a_\epsilon^2 - a_\epsilon + b_\epsilon^2 - b_\epsilon + 2), \notag\\
&\psi_2=(21a_\epsilon^4 + (30b_\epsilon - 21)a_\epsilon^3 + (42b_\epsilon^2 - 51 b_\epsilon + 52)a_\epsilon^2 + (30b_\epsilon^2 - 51 b_\epsilon + 60)a_\epsilon b_\epsilon \notag\\
& \qquad + (21b_\epsilon^2 - 21 b_\epsilon + 52)b_\epsilon^2)/10, \notag\\
&	\psi_3=(11a_\epsilon^5 + (21b_\epsilon - 11)a_\epsilon^4 + (32b_\epsilon^2 - 32b_\epsilon + 32)a_\epsilon^3 + (32b_\epsilon^2 - 42b_\epsilon  + 52)a_\epsilon^2b_\epsilon  \notag\\
&\qquad  + (21b_\epsilon^2 - 32b_\epsilon + 52)a_\epsilon b_\epsilon^2 + (11b_\epsilon^2 - 11 b_\epsilon + 32)b_\epsilon^3)/20,\notag\\
&	\psi_4=(14a_\epsilon^6 + (33b_\epsilon - 14)a_\epsilon^5 + (56b_\epsilon^2 - 47b_\epsilon  + 47)a_\epsilon^4 + (66b_\epsilon^2 - 75b_\epsilon  + 96)a_\epsilon^3b_\epsilon \notag\\ 
&\qquad + (52b_\epsilon^2 - 75b_\epsilon  + 126)a_\epsilon^2b_\epsilon^2 + (33b_\epsilon^2 - 43b_\epsilon  + 96)a_\epsilon b_\epsilon^3 + (10b_\epsilon^2 - 10 b_\epsilon + 39)b_\epsilon^4)/120, \notag\\
&	\psi_5=(5a_\epsilon^7 + (14b_\epsilon - 5 )a_\epsilon^6 + (26b_\epsilon^2 - 19b_\epsilon + 19)a_\epsilon^5 + (36b_\epsilon^2 - 35b_\epsilon + 47)a_\epsilon^4b_\epsilon  \notag\\
&\qquad + (31b_\epsilon^2 - 43b_\epsilon + 74)a_\epsilon^3b_\epsilon^2 + (22b_\epsilon^2 - 32b_\epsilon+ 76)a_\epsilon^2b_\epsilon^3  + (10b_\epsilon^2 - 10b_\epsilon + 39)a_\epsilon b_\epsilon^4 + 9b_\epsilon^5)/240.\notag
\end{align}
Utilizing $1\le a_\epsilon$, $1\le b_\epsilon$ and the simple fact:  $\mu_2 b_\epsilon^2 +\mu_1b_\epsilon + \mu_0>0$ for any $b_\epsilon \in \R$ if  $\mu_2>0$ and $\mu_1^2-4\mu_2\mu_0<0$, we deduce
\be\label{positive}
\psi_0 \ge 24, \quad \text{and} \quad	\psi_k \ge 0,\qquad k=1,\dots, 5.
\ee
Based on \eqref{L:tilde:2}--\eqref{positive} 
\be\label{greater:than:24}
	\mathcal{\tilde{L}}_h(\theta)_{i,j}=	\sum_{p=0}^6 \phi_ph^p\ge 	\sum_{p=0}^5 \psi_ph^p +\phi_6h^6 \ge 24, \quad \text{for any }  h>0 \text{ and } (x_i,y_j)\in \Omega.
\ee

\textbf{Step 4 (the sixth-order convergence rate):}
From $\Omega=(0,1)^2$ and \eqref{xiyj}, 
we define 
\[
\begin{split}
& \Omega_h:=\{(x_i,y_j):\ (x_i,y_j) \in \Omega\}, \qquad \overline{\Omega}_h:=\{(x_i,y_j):\ (x_i,y_j) \in \overline{\Omega}\}, \qquad \partial \Omega_h:=\overline{\Omega}_h \cap \partial \Omega,\\
& V(\overline{\Omega}_h):=\{(v)_{i,j}=v(x,y)|_{(x,y)=(x_i,y_j)}\in \mathbb{R}:\  (x_i,y_j) \in \overline{\Omega}\}. 
\end{split}
\]
Now the Dirichlet boundary condition $u=g$ on $\partial \Omega$, $\mathcal{L}_h$ in \eqref{FDMs:u:interior}, and $\mathcal{\tilde{L}}_h$ in \eqref{L:tilde} imply that our proposed FDM can be rewritten as: find $u_h \in
V(\overline{\Omega}_h)$ such that:
\be\label{DeltahuhF}
\mathcal{\tilde{L}}_h (u_h)_{i,j}=-F_{i,j}
 \quad \mbox{for}\quad (x_i,y_j)\in \Omega_h \quad\mbox{with}\quad u_h=g\quad \mbox{on}\quad
\partial \Omega_h.
\ee
According to \eqref{M:Matrix} and \eqref{Ckl:sum},
 the discrete maximum principle: 
\be\label{max:principle} 
\max_{(x_i,y_j)\in \Omega_h} v(x_i,y_j)\le \max_{(x_i,y_j)\in \partial \Omega_h} v(x_i,y_j), \quad \text{if} \quad \mathcal{\tilde{L}}_h v\ge 0 \ \text{in} \ \Omega_h, \ \text{for any} \ v\in V(\overline{\Omega}_h),
\ee
is proved by contradiction in the following \eqref{vmn:assume}--\eqref{max:prin:proof}:
Suppose that $\max\limits_{(x_i,y_j)\in \Omega_h} (v)_{i,j}> \max\limits_{(x_i,y_j)\in \partial \Omega_h} (v)_{i,j}$.
So we assume that
\be\label{vmn:assume}
(v)_{m,n}= v(x_m,y_n)=\max\left\{ v(x_i,y_j) :  (x_i,y_j)\in \Omega_h\right\}.
\ee
From  \eqref{Ckl:sum}
\[
\sum_{k,\ell\in\{-1,0,1\} \atop k\ne 0, \  \ell\ne 0} -C_{k,\ell} =  C_{0,0}.
\]
By \eqref{M:Matrix} 
\[
-C_{k,\ell}\ge  0 \quad \text{if} \quad (k,\ell)\ne (0,0), \quad \text{and} \quad  C_{0,0}>  0.
\]
Then
\be\label{Ckl:sign:ineq:1}
\sum_{k,\ell\in\{-1,0,1\} \atop k\ne 0, \  \ell\ne 0} -C_{k,\ell} (v)_{m+k,n+\ell} \le C_{0,0} (v)_{m,n}.
\ee
Based on  $\mathcal{\tilde{L}}_h v\ge 0$ in $\Omega_h$, and the definition of $\mathcal{\tilde{L}}_h$ in \eqref{L:h:u:1} and \eqref{L:tilde}, we demonstrate that 
\be\label{Ckl:sign:ineq:2}
0\le h^2(\mathcal{\tilde{L}}_h v)_{m,n}=-\sum_{k,\ell =-1}^1 C_{k,\ell} (v)_{m+k,n+\ell}=-C_{0,0} (v)_{m,n}-\sum_{k,\ell\in\{-1,0,1\} \atop k\ne 0, \  \ell\ne 0} C_{k,\ell} (v)_{m+k,n+\ell}.
\ee
Now, \eqref{Ckl:sign:ineq:1} and \eqref{Ckl:sign:ineq:2} result in
\be\label{max:prin:proof}
C_{0,0} (v)_{m,n}\le \sum_{k,\ell\in\{-1,0,1\} \atop k\ne 0, \  \ell\ne 0} -C_{k,\ell} (v)_{m+k,n+\ell} \le C_{0,0} (v)_{m,n}.
\ee
Equality holds throughout and $v$ attains the maximum value at all its nearest neighbors of $(x_m,y_n)$. Applying the same argument to the neighbors in $\Omega_h$ and repeat this argument, we  can say that $v$ must be a constant which  contradicts \eqref{vmn:assume}.
So \eqref{vmn:assume}--\eqref{max:prin:proof} prove the discrete maximum principle \eqref{max:principle}.

 \eqref{truncation:error:h:6} and the defined $\mathcal{\tilde{L}}_h$ in  \eqref{L:tilde} indicate
	\be\label{Rxiyj}
	\mathcal{\tilde{L}}_h 	(u)_{i,j}=-F_{i,j}+R_{i,j}, \quad \text{with} \quad	|R_{i,j}|\le
	c_3 h^6,  \quad \text{for any }  h>0 \text{ and } (x_i,y_j)\in \Omega_h, 
	\ee
	where $R_{ij}$ is the truncation error of $	\mathcal{\tilde{L}}_h 	(u)_{i,j}$ at the grid point $(x_i,y_j)\in \Omega_h$ and $0<c_3$ is independent of $h$. It is clear from \eqref{DeltahuhF} that   
	\be\label{intersect:trunR}
\mathcal{\tilde{L}}_h 	(u- u_h)_{i,j}=R_{i,j} \quad \mbox{for}\quad (x_i,y_j)\in \Omega_h \quad \mbox{with}\quad (u- u_h)_{i,j}=0\quad \mbox{for}\quad (x_i,y_j)\in \partial \Omega_h.
	\ee
From \eqref{greater:than:24}, and $c_3>0$ in \eqref{Rxiyj}, we observe that
	\be\label{4CH6}
	\tfrac{1}{24}c_3h^6(\mathcal{\tilde{L}}_h  \theta)_{i,j}\ge 	
	c_3 h^6,\quad \text{for any }  h>0 \text{ and } (x_i,y_j)\in \Omega_h.
	\ee
	Let  
	\[
	(E_h)_{i,j}:= u(x_i,y_j)-(u_h)_{i,j} \quad \text{with} \quad(x_i,y_j)\in \overline{\Omega}_h. 
	\]
	Then
	\be\label{Eh:bd}
		(E_h)_{i,j}=0 \quad \text{for} \quad (x_i,y_j)\in \partial\Omega_h. 
	\ee
	It follows from \eqref{intersect:trunR}  that
	\be\label{order6:proof:compare}
	\begin{split}
	(\mathcal{\tilde{L}}_h  (E_h+	\tfrac{1}{24} c_3 h^6 \theta))_{i,j}
	&=(\mathcal{\tilde{L}}_h  E_h)_{i,j}+	\tfrac{1}{24} c_3 h^6 (\mathcal{\tilde{L}}_h  \theta)_{i,j}\\
	&= R_{i,j}+	\tfrac{1}{24} c_3 h^6 (\mathcal{\tilde{L}}_h  \theta)_{i,j}, \quad \text{for any }  h>0 \text{ and } (x_i,y_j)\in \Omega_h.
	\end{split}
	\ee
According to \eqref{Rxiyj}, \eqref{4CH6}, and \eqref{order6:proof:compare},
	\be\label{L:h:E:theta}
	(\mathcal{\tilde{L}}_h  (E_h+	\tfrac{1}{24} c_3 h^6 \theta))_{i,j}\ge 0, \quad \text{for any }  h>0 \text{ and } (x_i,y_j)\in \Omega_h.
	\ee
	Recall the definition of the comparison function $\theta$ in \eqref{L:tilde} with $a_\epsilon,b_\epsilon \ge 1$ and $\Omega=(0,1)^2$, $\theta$ satisfies
	\be \label{theta:max}
	0\le \theta \le a_\epsilon^2+b_\epsilon^2 \quad \text{on} \quad \overline{\Omega}_h.
	\ee
	As $c_3> 0$  in \eqref{Rxiyj}, 
	\[
		\max_{(x_i,y_j)\in \Omega_h} (E_h)_{i,j}
	\le
	\max_{(x_i,y_j)\in \Omega_h} (E_h+	\tfrac{1}{24} c_3 h^6 \Theta)_{i,j}.
	\]
	\eqref{max:principle}, \eqref{L:h:E:theta}, and $c_3>0$ yield
	\[
	\max_{(x_i,y_j)\in \Omega_h} (E_h)_{i,j} \le 	\max_{(x_i,y_j)\in \partial \Omega_h} (E_h+	\tfrac{1}{24} c_3 h^6 \theta)_{i,j}\le \max_{(x_i,y_j)\in \partial \Omega_h} (E_h)_{i,j}+
	\tfrac{1}{24} c_3 h^6 \times \max_{(x_i,y_j)\in \partial \Omega_h}(\theta)_{i,j}.
	\] 
We infer from \eqref{Eh:bd} and \eqref{theta:max} that	
	\[
		\max_{(x_i,y_j)\in \Omega_h} (u-u_h)_{i,j} \le Ch^6, \quad \text{with} \quad 	C=\frac{a_\epsilon^2+b_\epsilon^2}{24}  c_3, \quad \text{for any }  h>0.
	\]
Apply the similar technique to $-E_h$, we finally prove \eqref{convergence:order:6}.
\end{proof}
	\subsection{The sixth-order FDM for \bm{$a\le b $}}\label{sec:case2}
Similar to  \cref{theorem:interior}, we construct the sixth-order 9-point compact FDM for the model problem \eqref{Model:Original} with $a\le b $ in this section. By our calculation, we observe that $c_{k,\ell,p}$ with $p=0,\dots,3$ and $a\le b $ are exactly the same as those in \eqref{c:h:0}--\eqref{c:h:3}. So, we only provide the following notations for the simplification:\\
\textbf{coefficients of $h^4$ of the  LHS of the FDM:}
\begin{align}\label{case2:p:4}
	& c_{-1, -1, 4} := -(9 a_\epsilon ^4 + 40 a_\epsilon ^3 b_\epsilon   + 60 a_\epsilon ^2 b_\epsilon ^2  + 40 a_\epsilon b_\epsilon ^3   + 11 b_\epsilon ^4)/240, \notag \\
	& c_{-1, 0, 4} := -(21 a_\epsilon ^4 + 43  a_\epsilon ^3 b_\epsilon  + 33 a_\epsilon ^2  b_\epsilon ^2 + 13 a_\epsilon b_\epsilon ^3   + 2 b_\epsilon ^4)/120 ,    \notag \\
	&c_{-1, 1, 4} := -(9 a_\epsilon ^4 +  b_\epsilon ^4)/240, \quad c_{0, -1, 4} := -(13 a_\epsilon ^3 b_\epsilon   + 33 a_\epsilon ^2 b_\epsilon ^2  + 43 a_\epsilon b_\epsilon ^3   + 23 b_\epsilon ^4)/120 ,\\
	& c_{0, 0, 4} := (9a_\epsilon ^4  + 13 b_\epsilon ^4 )/40+ 4 a_\epsilon ^2 b_\epsilon ^2 /5+ 19 (a_\epsilon ^3 b_\epsilon    +a_\epsilon b_\epsilon ^3 ) /30 ,\quad c_{0, 1, 4} := 0,  \notag \\
	& c_{1, -1, 4} :=  (a_\epsilon ^4 - 11 b_\epsilon ^4)/240,\quad  c_{1, 0, 4} :=  (a_\epsilon ^4 -  b_\epsilon ^4)/60, \quad c_{1, 1, 4} :=  (a_\epsilon ^4 -  b_\epsilon ^4)/240,  \notag
\end{align}
\textbf{coefficients of $h^5$ of the  LHS of the FDM:}
\begin{align}\label{case2:p:5}
	& c_{-1, -1, 5} := - a_\epsilon b_\epsilon  r_1 ( a_\epsilon ^2 +  a_\epsilon  b_\epsilon  +  b_\epsilon ^2)/24, \notag \\
	& c_{-1, 0, 5} := -(11 a_\epsilon ^5 +5 a_\epsilon ^4 b_\epsilon  +25 a_\epsilon ^3 b_\epsilon ^2 +10 a_\epsilon ^2 b_\epsilon ^3 +21 b_\epsilon ^5)/240, \quad c_{-1, 1, 5} := 0, \notag \\
	& c_{0, -1, 5} := r_1(    5 b_\epsilon -a_\epsilon )( 2a_\epsilon ^3 - 6a_\epsilon ^2  b_\epsilon  +  a_\epsilon  b_\epsilon ^2 - 6 b_\epsilon ^3)/240, \\
	& c_{0, 0, 5} := (18 a_\epsilon ^5 - 29  a_\epsilon ^4 b_\epsilon + 65a_\epsilon ^3  b_\epsilon ^2 + 45 a_\epsilon ^2 b_\epsilon ^3 + 19 a_\epsilon  b_\epsilon ^4 + 86 b_\epsilon ^5)/240, \notag \\
	& c_{0, 1, 5} := 	4\beta, \quad  c_{1, -1, 5} := 	\beta, \quad c_{1, 0, 5} := 	4\beta, \quad c_{1, 1, 5} := 	\beta,\notag
\end{align}
where
\[
\beta:= r_1( b_\epsilon  -  a_\epsilon )( a_\epsilon ^3 - 6 a_\epsilon ^2 b_\epsilon  + 2 a_\epsilon  b_\epsilon ^2 - 7 b_\epsilon ^3)/480,
\]
\textbf{coefficients of $h^6$ of the  LHS of the FDM:}
\begin{align}\label{case2:p:6}
	& c_{-1, -1, 6} := (-20 a_\epsilon ^6  +63 a_\epsilon ^5 b_\epsilon   + 2 a_\epsilon ^4 b_\epsilon ^2 - 199 a_\epsilon ^3 b_\epsilon ^3 - 96 a_\epsilon ^2 b_\epsilon ^4 - 188 a_\epsilon  b_\epsilon ^5 - 210 b_\epsilon ^6)/1920, \notag \\
	&  c_{-1, 0, 6} := 0, \quad c_{-1, 1, 6} := (-5 a_\epsilon ^5 b_\epsilon  + 22 a_\epsilon ^4 b_\epsilon ^2 - 27 a_\epsilon ^3 b_\epsilon ^3 - 16 a_\epsilon ^2 b_\epsilon ^4 - 4 a_\epsilon  b_\epsilon ^5 - 42 b_\epsilon ^6)/1920, \notag \\
	&  c_{0, -1, 6} := (5 a_\epsilon ^6 - 27 a_\epsilon ^5 b_\epsilon  + 49 a_\epsilon ^4 b_\epsilon ^2 - 11 a_\epsilon ^3 b_\epsilon ^3 - 12 a_\epsilon ^2 b_\epsilon ^4 + 38 a_\epsilon  b_\epsilon ^5 - 42 b_\epsilon ^6)/480,  \\
	&c_{0, 0, 6} :=\lambda/96,\quad  c_{0, 1, 6} :=-\lambda/480, \quad  c_{1, -1, 6} :=-\lambda/1920, \quad  c_{1, 0, 6} :=-\lambda/480, \notag \\
	&  c_{1, 1, 6} :=-\lambda/1920, \notag  
\end{align}
where
\[
\lambda:=5 a_\epsilon ^5 b_\epsilon  - 22 a_\epsilon ^4 b_\epsilon ^2 + 27 a_\epsilon ^3 b_\epsilon ^3 + 16 a_\epsilon ^2 b_\epsilon ^4 + 4 a_\epsilon  b_\epsilon ^5 + 42 b_\epsilon ^6, 
\]
\textbf{the  RHS of the FDM:}
\be
\begin{split}\label{case2:Fij}
	F_{i,j}	&:=6f_\epsilon+3h	r_{1} f_\epsilon + h^2 ((21r_2 + 30r_3)f_\epsilon/20 - 	\bv \cdot \nabla f_\epsilon/2 + \Delta f_\epsilon/2) \\
	& \qquad +{h^3r_{1}}	((11 r_2/40 + r_3/4 ) f_{_\epsilon} -  (	\bv \cdot \nabla f_{_\epsilon} - \Delta f_{_\epsilon})/4)\\
	&\qquad + h^4 ( r_{14} f_\epsilon  +r_{5} f_{_\epsilon}^{(1, 0)}  +r_{6} f_{_\epsilon}^{(0, 1)}+r_{7} f_{_\epsilon}^{(2, 0)}+ r_{3}  f_{_\epsilon}^{(1, 1)}/15+ r_{8} f_{_\epsilon}^{(0, 2)}-  2\zeta )\\
	&\qquad + h^5 (  r_{15} f_\epsilon +r_{10} f_{_\epsilon}^{(1, 0)}  +r_{11} f_{_\epsilon}^{(0, 1)}+r_{12} f_{_\epsilon}^{(2, 0)}+ r_{1}r_{3}  f_{_\epsilon}^{(1, 1)}/30+ r_{13} f_{_\epsilon}^{(0, 2)} \\
	&\quad \qquad \quad -  r_{1}\zeta+ r_{1}( \Delta^2  f_{_\epsilon} +2f_{_\epsilon}^{(2, 2)} )/120),
\end{split}	
\ee
where $r_1,\dots r_{15}, \bv,  \zeta $ are defined in \eqref{r1:r15}, \eqref{zeta}, and each $f_\epsilon^{(m,n)}$ is defined in \eqref{fmn}.
\begin{theorem}\label{theorem:2:interior} Assume that the solution $u$, the source term $f$ and the constant coefficient $(a,b)$ of \eqref{Model:Original} satisfy that $u\in C^8(\overline{\Omega})$, $f\in C^6(\overline{\Omega})$, $a\le b$. Let the grid point $(x_i,y_j)\in \Omega$ and define the following 9-point compact FDM:
	\be\label{FDMs:2:u:interior}
	\mathcal{L}_h (u_h)_{i,j} :=\frac{1}{h^2}\sum_{k,\ell=-1}^1 C_{k,\ell} (u_h)_{i+k,j+\ell}=F_{i,j}\Big|_{(x,y)=(x_i,y_j)} \quad \text{with} \quad C_{k,\ell}:=\sum_{p=0}^6 c_{k,\ell ,p}h^p,
	\ee
	where $c_{k,\ell,p}$ is defined in \eqref{c:h:0}--\eqref{c:h:3}, \eqref{case2:p:4}--\eqref{case2:p:6} and $F_{i,j}$ is defined in \eqref{case2:Fij}. Then
	\[
	C_{k,\ell}\le  0 \quad \text{if} \quad (k,\ell)\ne (0,0), \quad   C_{0,0}>  0, \quad  \text{for any mesh size } h>0, \quad \text{and} \quad \sum_{k,\ell=-1}^1 C_{k,\ell}=0.
	\]
	Furthermore,   $\mathcal{L}_h (u_h)_{i,j}$ in \eqref{FDMs:2:u:interior} with $u=g$ on $\partial \Omega$  forms an M-matrix for any mesh size $h$,  and
	\[
	\|u-u_h\|_{\infty}\le Ch^6, \quad  \text{for any mesh size } h>0,
	\]
	where  $C$ is independent of $h$.
\end{theorem}
\begin{proof}
Similar to the proof of \cref{theorem:interior}, so we omit the proof.
\end{proof}
	\section{Numerical experiments}\label{numerical:test}
As the exact solution is not available, we verify the sixth-order convergence rate of the proposed FDM  for the 2D transport problem \eqref{Model:Original}  by the following $l_\infty$ norm:
\begin{align*}
	& \|u_h-u_{h/2}\|_{\infty}:=\max_{0\le ij\le N} |(u_h)_{i,j} -(u_{h/2})_{2i,2j}|, \qquad \text{with}\quad  h=1/N.
\end{align*}
%
%

\begin{example}\label{example1}
	\normalfont
	Let  $\Omega=(0,1)^2$. Then
	the functions in \eqref{Model:Original} are given by
	\begin{align*}
		&	\epsilon=10^{-2},\qquad f=\sin(\pi x)\sin(\pi y),\qquad  (a,b)=(1,10^{-2}), \ (10^{-2},1), \qquad g=0.
	\end{align*}
	The numerical results are presented in \cref{tab:1:example,fig:1:example,fig:2:example}.
\end{example}
\begin{table}[htbp]
	\caption{Performance in \cref{example1}  of the proposed FDM.}
	\centering
	\setlength{\tabcolsep}{4mm}{
		\begin{tabular}{c|c|c|c|c}
			\hline
			\multicolumn{1}{c|}{}  &
			\multicolumn{2}{c|}{$(a,b)=(1,10^{-2})$}  &
			\multicolumn{2}{c}{$(a,b)=(10^{-2},1)$} \\
			\hline
			$h$
			&  $\|u_h-u_{h/2}\|_{\infty}$
			&order &  $\|u_h-u_{h/2}\|_{\infty}$ & order \\
			\hline
$1/2^4$  &  2.24815E-01  &    &  2.64991E-01  &  \\
$1/2^5$  &  6.30645E-02  &  1.83  &  2.38346E-01  &  0.15\\
$1/2^6$  &  4.24966E-03  &  3.89  &  3.53978E-02  &  2.75\\
$1/2^7$  &  1.02076E-04  &  5.38  &  9.62929E-04  &  5.20\\
$1/2^8$  &  2.00464E-06  &  5.67  &  1.90399E-05  &  5.66\\
$1/2^9$  &  3.48531E-08  &  5.85  &  3.31949E-07  &  5.84\\
$1/2^{10}$  &  5.77342E-10  &  5.92  &  5.47505E-09  &  5.92\\
$1/2^{11}$  &  5.51370E-12  &  6.71  &  8.75092E-11  &  5.97\\
			\hline
	\end{tabular}}
	\label{tab:1:example}
\end{table}	
\begin{figure}[htbp]
	\centering
	\begin{subfigure}[b]{0.3\textwidth}
		\includegraphics[width=5cm,height=5cm]{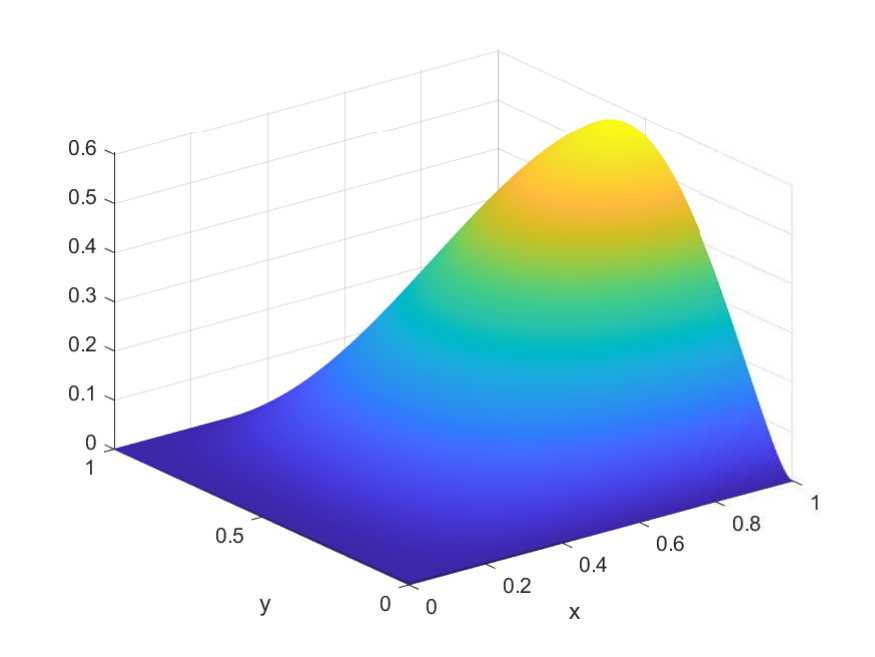}
	\end{subfigure}	
	\begin{subfigure}[b]{0.3\textwidth}
		\includegraphics[width=5cm,height=5cm]{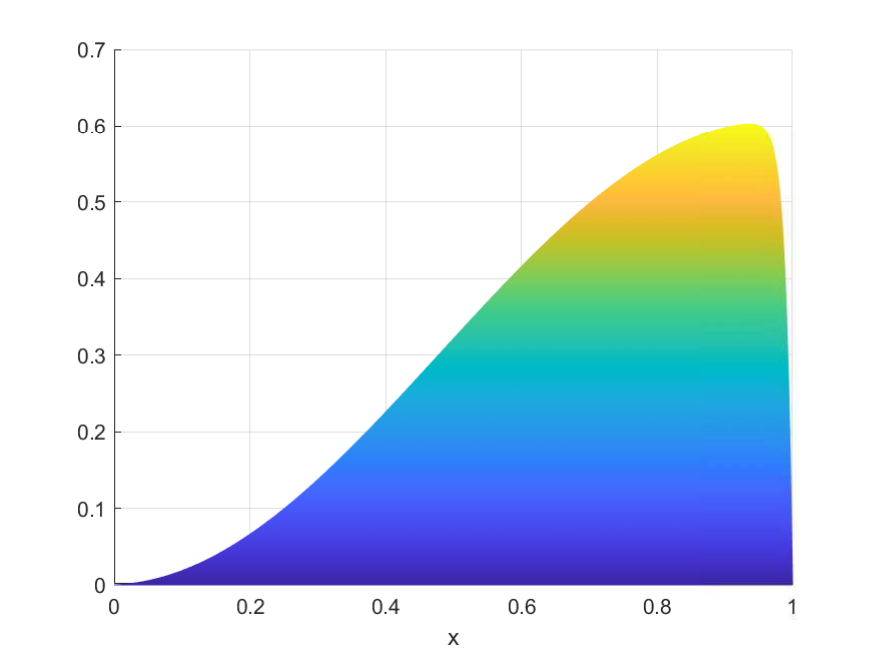}
	\end{subfigure}
	\begin{subfigure}[b]{0.3\textwidth}
		\includegraphics[width=5cm,height=5cm]{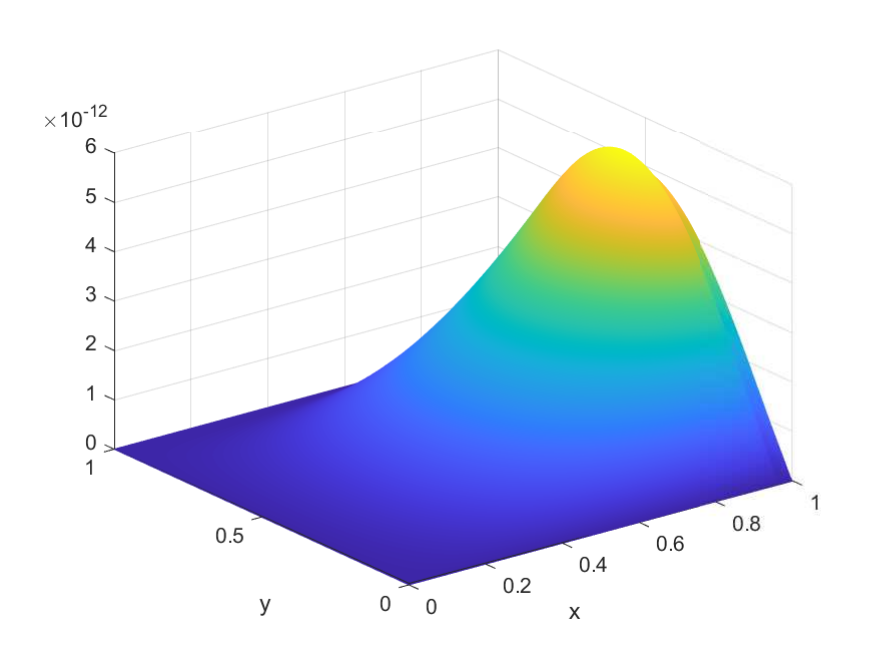}
	\end{subfigure}
	\caption{\cref{example1}: $u_{h/2}$ (left and middle),  and $|u_h-u_{h/2}|$ (right) at all grid points in $\overline{\Omega}=[0,1]^2$ with $h=\tfrac{1}{2^{11}}$ and $(a,b)=(1,10^{-2})$.}
	\label{fig:1:example}
\end{figure}	
\begin{figure}[htbp]
	\centering
	\begin{subfigure}[b]{0.3\textwidth}
		\includegraphics[width=5cm,height=5cm]{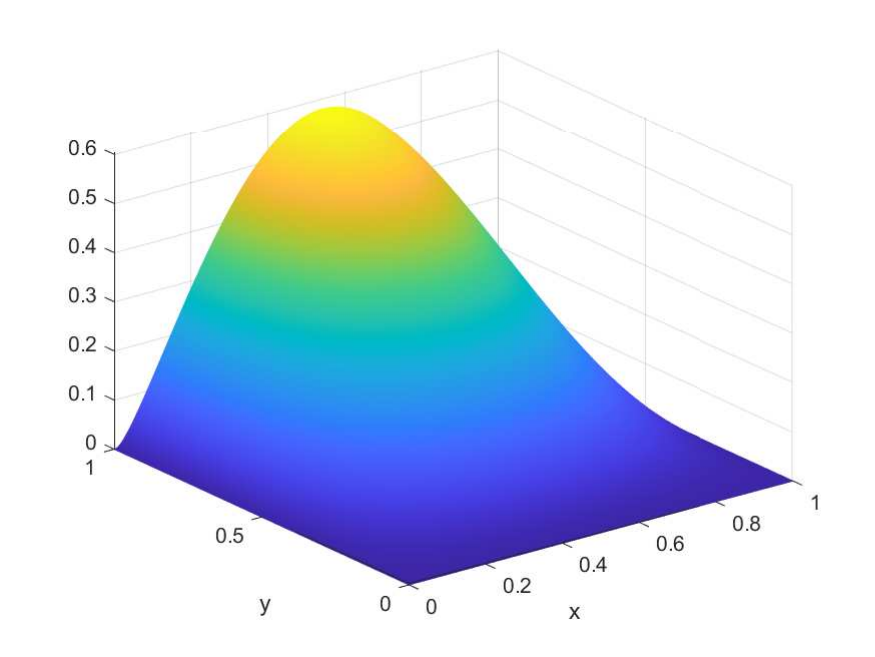}
	\end{subfigure}	
	\begin{subfigure}[b]{0.3\textwidth}
		\includegraphics[width=5cm,height=5cm]{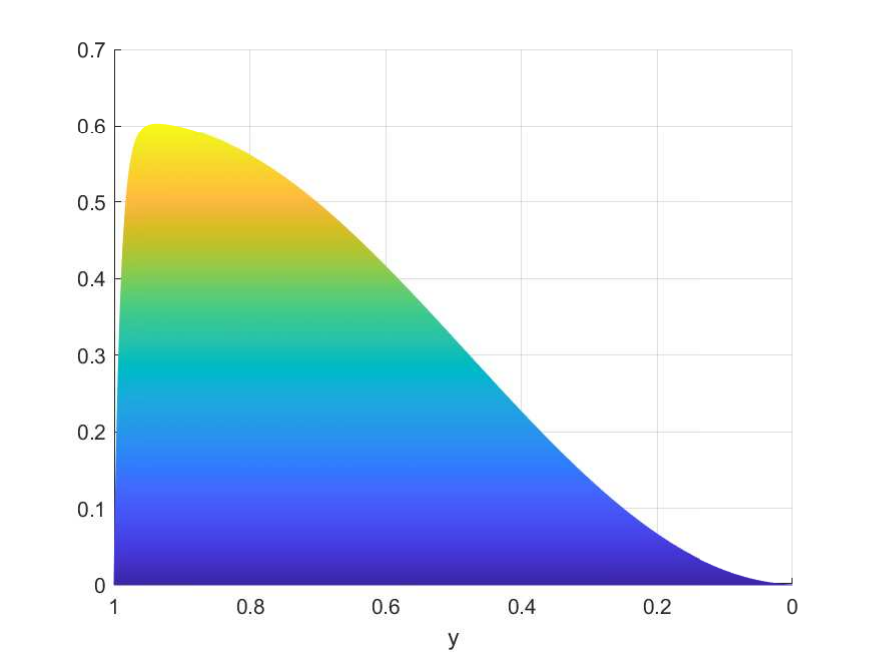}
	\end{subfigure}
	\begin{subfigure}[b]{0.3\textwidth}
		\includegraphics[width=5cm,height=5cm]{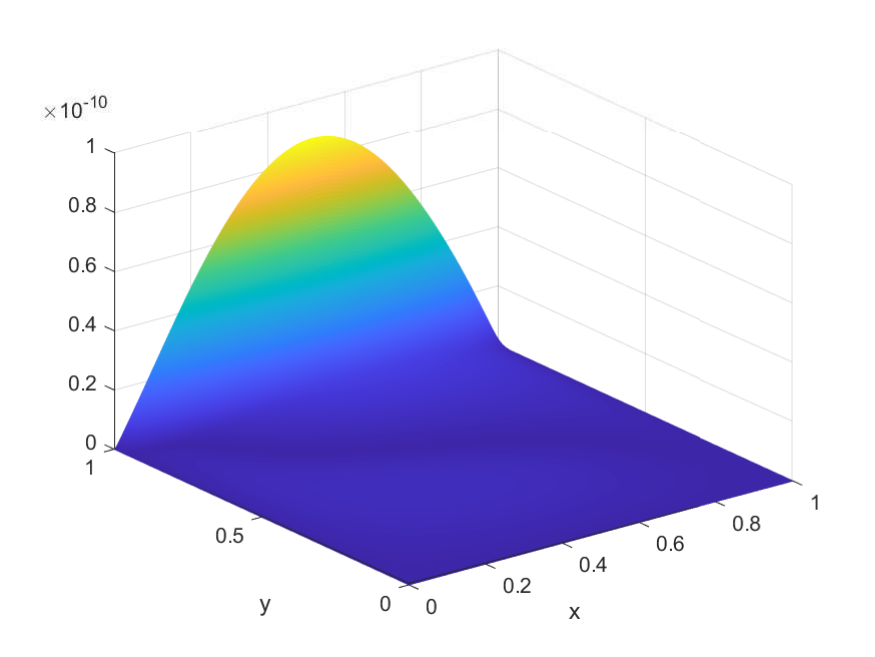}
	\end{subfigure}
	\caption{\cref{example1}:  $u_{h/2}$ (left and middle),  and $|u_h-u_{h/2}|$ (right) at all grid points in $\overline{\Omega}=[0,1]^2$ with $h=\tfrac{1}{2^{11}}$ and $(a,b)=(10^{-2},1)$.}
	\label{fig:2:example}
\end{figure}	
	\section{Contribution}\label{sec:contr}
	To the best of our knowledge, no compact FDM has been proposed in the literature that can	rigorously prove sixth-order convergence rate of the maximum pointwise error  for the 2D transport problem \eqref{Model:Original} for any mesh size $h$. In this paper, we propose a sixth-order compact 9-point FDM for the 2D transport problem with the constant coefficient and the Dirichlet boundary condition using the uniform mesh in a unit square. The matrix of our FDM is an M-matrix for any mesh size $h$ and the expression of the stencil is provided explicitly. With the aid of the explicit formula, the proposed FDM can be easily implemented to reproduce our numerical results and verify the sixth-order pointwise accuracy. Furthermore, we also construct the explicit formula of the comparison function to theoretically prove the sixth-order convergence rate of the error in the  $l_\infty$ norm by the discrete maximum principle. Remarkably, the proof of the sixth-order convergence does not require any constraints on the mesh size $h$. In our future work, we aim to improve the performance of our FDM by incorporating the optimal FDM technique from \cite{Layton1990,Layton1993} and the nonuniform mesh strategy from \cite{Ge2011}.

\noindent\textbf{Acknowledgment}

The author gratefully acknowledges William Layton (wjl@pitt.edu, Department of Mathematics, University of Pittsburgh, Pittsburgh, PA 15260 USA) for introducing this valuable two-dimensional transport problem, which served as the inspiration for the present paper.

\end{document}